\documentclass[runningheads]{llncs}

\usepackage[T1]{fontenc}
\usepackage[utf8]{inputenc}
\usepackage[english]{babel}

\usepackage{graphicx}
\usepackage{amssymb}
\usepackage{amsmath}

\newtheorem{assumption}{Assumption}

\usepackage{booktabs}
\usepackage{tabularx}
\usepackage{makecell}

\usepackage{hyperref}

\usepackage{algorithm}
\usepackage{algpseudocode}

\usepackage{xcolor}

\begin{document}
	
	\title{An Inexact Augmented Lagrangian Method for $(L_0,L_1)$-Smooth Convex Optimization}
	
	\author{
		A.A. Vyguzov\inst{1,3,4}\orcidID{0009-0005-1681-1750} 
		\and
		F.\;S.\;Stonyakin\inst{1,2}\orcidID{0000-0002-9250-4438}
	}
	
	\authorrunning{A. Vyguzov et al.}
	
	\institute{
		Moscow Institute of Physics and Technology, Dolgoprudny, Institutsky lane, 9, Russia 
		\and
		V.~I.~Vernadsky Crimean Federal University, Simferopol, Academician Vernadsky Avenue, 4,, Republic of Crimea, Russia 
		\and
		Innopolis University, Kazan, Tatarstan, 420500, Russia 
		\and
		Adyghe State University, Maikop, Russia, Pervomayskaya Str., 208
	}
	
	\maketitle
	
	\begin{abstract}
		Augmented Lagrangian methods are among the most effective approaches for solving constrained convex optimization problems. However, classical complexity analyses of first-order methods applied within the augmented Lagrangian framework usually rely on the assumption that the objective function has a Lipschitz continuous gradient. This assumption excludes an important class of generalized smooth functions whose gradients may grow unboundedly.
		
		In this paper, we study an inexact augmented Lagrangian method for solving linearly constrained convex optimization problems with $(L_0,L_1)$-smooth objective functions. We show that the augmented Lagrangian subproblems preserve the $(L_0,L_1)$-smooth structure, with parameters depending on the penalty coefficient. This property allows us to employ recent accelerated first-order schemes designed for generalized smooth optimization instead of classical smooth optimization methods. In particular, we combine the inexact augmented Lagrangian framework with a two-stage acceleration procedure based on clipped gradient descent and accelerated optimization.
		
	\end{abstract}
	
	\section{Introduction}
	
	In this paper, we consider the linearly constrained convex optimization problem
	\begin{align}
		\min_{x \in Q} \; f(x)
		\quad \text{s.t.} \quad
		Ax=b,
		\label{eq:main_problem}
	\end{align}
	where $Q$ is a convex compact set, $A \in \mathbb{R}^{m \times n}$, $b \in \mathbb{R}^{m}$, and $f$ is a convex $(L_0^f,L_1^f)$-smooth function.
	
	Let us recall the definition of $(L_0,L_1)$-smoothness, which was first introduced in \cite{zhang2019gradient}.
	\begin{definition}
		A function $f$ is called \emph{$(L_0,L_1)$-smooth} if, for all $x$, it satisfies
		\begin{equation}\label{def:l0_l1_twice_dif}
			\| \nabla^2 f(x) \| \leq L_0 + L_1 \| \nabla f(x) \|,
		\end{equation}
		for some constants $L_0, L_1 > 0$. Throughout this paper, unless stated otherwise, we use the standard Euclidean norm $\|\cdot\|$ for vectors and the spectral norm $\|\cdot\|$ for matrices.
	\end{definition}
	
	Initially $(L_0,L_1)$-smoothness was used to explain the superior convergence behavior of clipped gradient descent compared to standard gradient descent (GD) in deep learning applications. Subsequently the development of this notion go far beyond just clipping gradient descent and included gradient descend with Polyak step size, Adaptive gradient descent, $(L_0, L_1)$-Similar Triangles Method (\cite{gorbunov2024methods}), Frank-Wolfe method (\cite{vyguzov2025frank}), Random Coordinate Descent (\cite{lobanov2024linear}),  accelerated algorithms was introduced in \cite{vankov2025optimizing}, \cite{tyurin2025near}, stochastic setting was developed in \cite{zhang2020improved}, \cite{koloskova2023revisiting}. Also generalizations of $(L_0, L_1)$ was introduced in \cite{li2023convex}, \cite{chen2023generalized}. So we have strong motivation to study this notion on constraint optimization problems.
	
	Some examples of $(L_0,L_1)$-smooth functions are 
	\begin{itemize}
		\item $f(x) = \|x\|^n$, where $n$ is a positive integer, with $L_0 = 2n$ and $L_1 = 2n - 1$;
		\item $f(x) = \exp(a^\top x)$, where $L_0 = 0$ and $L_1 = \|a\|$;
		\item the logistic loss $f(x) = \log(1 + \exp(-a^\top x))$, where $a \in \mathbb{R}^d$, for which $L_0 = 0$ and $L_1 = \|a\|$, while the standard smoothness constant is $L = \|a\|^2$, which is typically much larger than $L_1$.
	\end{itemize}
	
	It is worth noting that when $L_1=0$, then $(L_0,L_1)$-smooth translate into ordinary $L$-smoothness, therefore it is more broad function family.
	
	Stated optimization problems \eqref{eq:main_problem} have a great interest. For example distributed optimization problems \cite{boyd2011distributed} has linear constraints for consensus projection, portfolio optimization problems and many others. 
	
	One of the methods to solve such problems include gradient method with projection, but the projection operator may have heavy computational costs. Also one can solve dual problem but it is not guaranteed that the new problem will be simpler then the primal. A classical approach for solving such problems is the Augmented Lagrangian Method (ALM), which transforms the original constrained problem into a sequence of unconstrained subproblems. More specifically, at an iteration $k$, for optimization problem \ref{eq:main_problem} ALM considers the augmented Lagrangian
	\begin{align}
		\mathcal{L}_{\beta_k}(x,y)
		=
		f(x)
		+
		y^\top(Ax-b)
		+
		\frac{\beta_k}{2}\|Ax-b\|^2,
		\label{eq:augmented_lagrangian}
	\end{align}
	where $\beta_k>0$ is the penalty parameter and $y$ denotes the dual variable. Typically, the penalty coefficient $\beta_k$ is increased throughout the iterations in order to enforce feasibility of the constraints and $x$ and $y$ variables updates alternatively. 
	
	In this paper we study Augmented Lagrangian method for the cases when target functions possess $(L_0, L_1)$-smoothness. We find that the augmented Lagrangian \eqref{eq:augmented_lagrangian} is $(L_0,L_1)$-smooth on the specific set. We use the Two-Stage Acceleration Procedure proposed in paper \cite{vankov2025optimizing} for each Augmented Lagrangian subproblem because it is well-suited algorithm for our problem because it is monotonically decreasing and has near optimal convergence rate. For the convergence analyzes framework we rely on inexact augmented Lagrangian method (iALM) proposed in \cite{xu2021iteration}. As a result, we obtain an ergodic convergence rate of $\mathcal{O}(\frac{K}{\varepsilon}) + \mathcal{O}(K)$, where $K$ denotes the number of outer iterations of iALM. 
	
	
	The main contributions of this work are as follows:
	\begin{itemize}
		\item We show that augmented Lagrangian function for optimization problem \eqref{eq:main_problem} is $(L_0,L_1)$-smooth function. Moreover, the parameter $L_0$ increase when penalty multiplier $\beta_k$ increase.
		\item We obtain an iteration complexity of our algorithm $\mathcal{O}(\frac{K}{\varepsilon}) + \mathcal{O}(K)$, where $K$ denotes the number of outer iterations, i.e. the number of individual subproblems with a parameter $\beta_k$ and dual multiplier $y^k$
		\item We show that constant penalty parameter when $\beta_0 = \beta_1 = ... = \beta_k$ is the best strategy with respect to number of iterations, but with respect to ill-conditioning it could be better to choose increasing sequence, for example $\beta_{k+1} = \beta_g\sigma^k$, where $\beta_g > 0$ is some constant.
	\end{itemize}
	
	\section{Minimization scheme}
	
	To solve problem \eqref{eq:main_problem}, we employ the Inexact Augmented Lagrangian Method (iALM)~\ref{alg:iALM}.
	
	\begin{algorithm}[t]
		\caption{Inexact Augmented Lagrangian Method with a Two-Stage Acceleration Procedure for Solving Subproblems}
		\label{alg:iALM}
		\begin{algorithmic}[1]
			\Require Initial primal point $x^0 \in Q$, dual variable $y^0 \in \mathbb{R}^m$, $y^0 = 0$ penalty parameters $\{\beta_k\}_{k \geq 0}$, accuracy sequence $\{\varepsilon_k\}_{k \geq 0}$.
			\For{$k=0,1,\ldots,K-1$}
			\State Let $x^{k+1}$ be an approximate solution computed by Algorithm~\ref{alg:two_stage}, satisfying
			\begin{align*}
				\mathcal{L}_{\beta_k}(x^{k+1},y^k)
				-
				\min_{x\in Q}
				\mathcal{L}_{\beta_k}(x,y^k)
				\leq
				\varepsilon_k .
			\end{align*}
			\State Update the dual variable
			\begin{align*}
				y^{k+1}
				=
				y^k+\beta_k(Ax^{k+1}-b).
			\end{align*}
			\EndFor
		\end{algorithmic}
	\end{algorithm}
	
	The key difference between iALM and the classical Augmented Lagrangian Method lies in the way the primal subproblems are solved. Standard ALM assumes that each augmented Lagrangian subproblem is solved exactly. In contrast, iALM allows inexact solutions with controllable accuracy $\varepsilon_k$. 
	
	As discussed in the introduction, the function $\mathcal{L}_{\beta_k}(\cdot,y^k)$ belongs to the class of convex $(L_0,L_1)$-smooth functions. Therefore, to solve it at each iteration, we employ the Two-Stage Acceleration Procedure proposed by \cite{vankov2025optimizing}. When the objective function is $L$-smooth, accelerated first-order methods such as Nesterov's Fast Gradient Descent (FGD) achieve optimal convergence guarantees and are therefore natural candidates for solving these subproblems. However, in our setting the augmented Lagrangian is generally not globally $L$-smooth. Consequently, standard accelerated methods cannot be applied directly, and the classical complexity analysis based on Lipschitz continuity of the gradient is no longer valid.
	
	Consider the gradient method (GM)
	\begin{equation}\label{eq:clipped_gm}
		x_{k+1}=x_k-\eta_k\nabla f(x_k),\qquad k\ge 0.
	\end{equation}
	
	In \cite{vankov2025optimizing}, three step-size choices are proposed that guarantee sufficient decrease of the function:
	
	\begin{align}
		\eta_k^* &= \frac{1}{L_1\|\nabla f(x_k)\|}\ln\!\left(1+\frac{L_1\|\nabla f(x_k)\|}{L_0+L_1\|\nabla f(x_k)\|}\right),\label{eq:eta_star}\\
		\eta_k^{\text{si}} &= \frac{1}{L_0+\frac32 L_1\|\nabla f(x_k)\|},\label{eq:eta_si}\\
		\eta_k^{\text{cl}} &= \min\left\{\frac{1}{2L_0},\frac{1}{3L_1\|\nabla f(x_k)\|}\right\}.\label{eq:eta_cl}
	\end{align}
	
	The first stage of the method is based on the clipped gradient descent method with one of the step sizes \eqref{eq:eta_star}, \eqref{eq:eta_si} or \eqref{eq:eta_cl}.
	
	Once the clipped gradient phase drives the iterate into the region where
	\begin{align}
		\mathcal{L}(x_k)-\mathcal{L}^\star \leq \frac{L_0}{5 L_1^2},
	\end{align}
	the objective becomes effectively $2L_0$-smooth. At this point one can switch to an accelerated optimization method. \cite{vankov2025optimizing} employ the accelerated AGMsDR algorithm, which enjoys the optimal complexity guarantees for smooth convex minimization.
	
	\begin{algorithm}[H]
		\caption{AGMsDR$(x_0,T(\cdot),L,K)$ \cite{nesterov2021primal}}
		\label{alg:agmsdr}
		\begin{algorithmic}[1]
			\Require Initial point $x_0 \in \mathbb{R}^d$, update rule $T(\cdot)$, constant $L>0$, number of iterations $K\ge1$
			\State $v_0=x_0,\quad A_0=0,\quad \zeta_0(x)=\frac12\|x-v_0\|^2$
			\For{$k=0,1,\ldots,K-1$}
			\State
			\[
			y_k=\arg\min_{y}\left\{f(y):\; y=v_k+\beta(x_k-v_k),\ \beta\in[0,1]\right\}
			\]
			\State $x_{k+1}=T(y_k)$ \Comment{one step of gradient method with sufficient decrease}
			\State Find $a_{k+1}>0$ from
			\[
			La_{k+1}^2=A_k+a_{k+1},
			\]
			and set
			\[
			A_{k+1}=A_k+a_{k+1}.
			\]
			\State
			\[
			v_{k+1}
			=
			\arg\min_{x\in\mathbb{R}^d}
			\left\{
			\zeta_{k+1}(x)
			:=
			\zeta_k(x)
			+
			a_{k+1}
			\bigl[
			f(y_k)
			+
			\langle\nabla f(y_k),x-y_k\rangle
			\bigr]
			\right\}.
			\]
			\EndFor
			\State \Return $x_K$
		\end{algorithmic}
	\end{algorithm}
	
	\begin{algorithm}[t]
		\caption{Two-Stage Acceleration Procedure (\cite{vankov2025optimizing}, Algorithm 6.2)}
		\label{alg:two_stage}
		\begin{algorithmic}[1]
			\Require Initial point $x_0$, target accuracy $\varepsilon$.
			\State Run clipped gradient descent until an iterate $\hat{x}$ satisfying
			\begin{align*}
				f(\hat{x})-f^\star \leq \frac{L_0}{5 L_1^2}
			\end{align*}
			is obtained.
			\State Use $\hat{x}$ as the starting point for Algorithm~\eqref{alg:agmsdr}.
			\State Run Algorithm~\eqref{alg:agmsdr} until
			\begin{align*}
				f(x)-f^\star \leq \varepsilon .
			\end{align*}
			\Return $x$.
		\end{algorithmic}
	\end{algorithm}
	
	\section{Ergodic convergence rate analyzes}\label{sect:erg_conv_rate_analyzes}

	\subsection{Basic facts}
	
	To analyze the convergence properties of Algorithm~\ref{alg:iALM}, we first recall several standard facts concerning the Karush--Kuhn--Tucker (KKT) conditions. Since problem \eqref{eq:main_problem} is convex and satisfies linear equality constraints, a primal-dual pair $(x^\star,y^\star)$ is optimal if and only if it satisfies the KKT system
	\begin{align}
		0 &\in \nabla f(x^\star) + A^\top y^\star + N_Q(x^\star),
		\label{eq:kkt_stationarity}\\
		Ax^\star-b &= 0,
		\label{eq:kkt_feasibility}
	\end{align}
	where $N_Q(x^\star)$ denotes the normal cone of the set $Q$ at $x^\star$.
	
	Throughout the paper, we assume that a KKT point $(x^\star,y^\star)$ exists. Under this assumption, strong duality holds and the optimal primal and dual objective values coincide. In particular, for any $x \in Q$ satisfying $Ax=b$, we have
	\begin{align}\label{eq:kkt_conv}
		f(x)-f(x^\star) + \langle y^\star,Ax-b\rangle \geq 0.
	\end{align}
	
	\subsection{Ergodic convergence rate analysis of outer iterations}
	
	In this section we assume that
	\begin{assumption}\label{assum:conv}
		We have the following assumptions
		\begin{itemize}
			\item KKT point exists
			\item Each point $x^k$ at each iteration of algorithm \ref{alg:iALM} exist
		\end{itemize}
	\end{assumption}
	
	Our goal is to find an $\varepsilon$-optimal solution:
	\begin{equation}
		f(x_k) - f(x^*) \leq \varepsilon \quad \quad, \quad \quad \| Ax^\star-b \| \leq \varepsilon.
	\end{equation}
	
	This section relies on the analysis from \cite{xu2021iteration}. We start with important theorem for our further analyzes.
	
%
%
	The next theorem presents the one-iteration progress estimate. 
	\begin{theorem}[\cite{xu2021iteration}, One-iteration progress of iALM]
		\label{thm:xu3}
		Assume that Assumption \ref{assum:conv} holds and $\{(x^k,y^k)\}$ be generated by Algorithm~\ref{alg:iALM}. Then for any feasible point $x\in Q$ satisfying $Ax=b$, and any $y\in\mathbb{R}^m$, the following inequality holds:
		\begin{align}
			&
			f(x^{k+1})-f(x)
			+\langle y,r^{k+1}\rangle
			+\frac{\beta_k-\rho_k}{2}\|r^{k+1}\|^2
			+\frac{1}{2\rho_k}\|y^{k+1}-y^k\|^2
			\nonumber\\
			&
			\leq
			\frac{1}{2\rho_k}\|y^k-y\|^2
			-
			\frac{1}{2\rho_k}\|y^{k+1}-y\|^2
			+\varepsilon_k ,
			\label{eq:one_step_progress}
		\end{align}
		where $r^k = Ax^k - b$.
	\end{theorem}

	This important result \eqref{thm:xu3} is used in proof of the next Theorem \ref{thm:xu4}, for more details of the proof see in \cite{xu2021iteration}. We also involve this result for Lemma \ref{lem:y_bnd} proof.
	
	\begin{theorem}[[\cite{xu2021iteration}, Ergodic convergence rate of iALM]
		\label{thm:xu4}
		Assume the Assumption~\ref{assum:conv} holds and let the sequence $\{(x^k,y^k)\}$ be generated by Algorithm~\ref{alg:iALM} with $y^0=0$ and $0<\rho_k\leq \beta_k$ for all $k$.
		
		Define
		\begin{equation}\label{ergodic_point_def}
			\overline{x}^{K}
			=
			\frac{\sum_{t=0}^{K-1}\rho_t x^{t+1}}
			{\sum_{t=0}^{K-1}\rho_t}.
		\end{equation}
		
		Then
		\begin{align}
			\left|
			f(\overline{x}^{K})-f(x^\star)
			\right|
			\leq
			\frac{1}{\sum_{k=0}^{K-1}\rho_k}
			\left(
			2\|y^\star\|^2
			+
			\sum_{k=0}^{K-1}\rho_k\varepsilon_k
			\right),
			\label{eq:ergodic_obj}
		\end{align}
		and
		\begin{align}
			\|A\overline{x}^{K}-b\|
			\leq
			\frac{1}{\sum_{k=0}^{K-1}\rho_k} 
			\left( 
			\frac{(1+\|y^\star\|)^2}{2}
			+
			\sum_{k=0}^{K-1}\rho_k\varepsilon_k 
			\right).
			\label{eq:ergodic_feas}
		\end{align}
	\end{theorem}
	
	Theorem~\ref{thm:xu4} shows that the convergence rate of iALM is determined by two quantities: the cumulative penalty parameter $\sum_{k=0}^{K-1}\rho_k$ and the accumulated inexactness $\sum_{k=0}^{K-1}\rho_k\varepsilon_k$. Therefore, to guarantee convergence of the method, the first quantity must grow with the number of iterations, while the second must remain bounded. Motivated by these requirements, we choose the parameters
	\begin{equation}\label{rho_epsilon_pick}
		\boxed{
			\sum_{k=0}^{K-1} \rho_k \varepsilon_k
			\leq
			\frac{C_\varepsilon}{2}
		}
	\end{equation}
	
	We further specify the penalty and dual stepsize parameters as
	\begin{equation}\label{beta_rho_pick}
		\boxed{
			\beta_k
			=
			\rho_k
			=
			\frac{C_\beta}{K\varepsilon}
		}
	\end{equation}
	for all $k=0,\ldots,K-1$, where $C_\beta>0$ is a constant, $\varepsilon$ is the desired accuracy of Algorithm~\ref{alg:iALM}, and $K$ is the total number of outer iterations.
	
	It remains to estimate the computational cost of solving each augmented Lagrangian subproblem to the prescribed accuracy $\varepsilon_k$. This is the focus of the next subsection. First, however, we establish the $(L_0,L_1)$-smoothness constants of the augmented Lagrangian~\eqref{eq:augmented_lagrangian}.
	
	\subsection{$(L_0,L_1)$-smoothness of the augmented Lagrangian with constraints $Ax=b$}\label{smoothness_al}
	
	In this subsection, we derive the $(L_0,L_1)$-smoothness constants of the augmented Lagrangian~\eqref{eq:augmented_lagrangian}. To this end, we first establish an upper bound on the norm of the dual variable $\|y^K\|$. This bound enables us to derive a uniform upper bound on the constant $L_0^{\mathcal{L}}$.
		
	\begin{lemma}[Upper bound of $\| y^K \|$]\label{lem:y_bnd}
		Let $\{(x^k,y^k)\}$ be the sequence generated by Algorithm~\ref{alg:iALM} with $\{\beta_k\}$ and $\{\rho_k\}$ chosen according to \eqref{beta_rho_pick}. Assume that $y^0=0$ and that the inexactness sequence satisfies
		\[
		\sum_{k=0}^{K-1}\rho_k\varepsilon_k
		\leq
		\frac{C_\varepsilon}{2}.
		\]
		Then
		\[
		\|y^K\|
		\leq
		2\|y^\star\|
		+
		\sqrt{C_\varepsilon}.
		\]
	\end{lemma}
	
	\begin{proof}
		Letting $(x,y)=(x^\star,y^\star)$ in Theorem~\ref{thm:xu3} and using \eqref{eq:kkt_conv}, we obtain
		\[
		\frac{1}{2}\|y^{k+1}-y^\star\|^2
		\leq
		\frac{1}{2}\|y^{k}-y^\star\|^2
		+
		\rho_k\varepsilon_k .
		\]
		
		Summing the above inequality from $t=0$ to $k-1$ yields
		\[
		\frac{1}{2}\|y^{k}-y^\star\|^2
		\leq
		\frac{1}{2}\|y^{0}-y^\star\|^2
		+
		\sum_{t=0}^{k-1}\rho_t\varepsilon_t ,
		\qquad
		0\leq k\leq K.
		\]
		
		Since $y^0=0$, it follows that
		\[
		\|y^{k}-y^\star\|
		\leq
		\sqrt{
			\|y^\star\|^2
			+
			2\sum_{t=0}^{k-1}\rho_t\varepsilon_t
		}.
		\]
		
		Next following the fact that $\| v \| \leq \| v \|_1$ for any vector $v$ and using triangle inequality we get
		
		\[
		\|y^{k}\| \leq 2 \|y^\star\| + \sqrt{C_\varepsilon}
		\]
		
		In particular,
		\[
		\|y^K\| \leq 2\|y^\star\| + \sqrt{C_\varepsilon}.
		\]
		This completes the proof.
	\end{proof}
	
	Next, we aim to establish that the Lagrangian is $(L_0,L_1)$-smooth and derive the corresponding smoothness constants.
	
	\begin{proposition}\label{prop:l0l1_obj}
		Let
		\[
		\mathcal{L}(x)
		=
		f(x)
		+
		y_k^\top(Ax-b)
		+
		\frac{\beta}{2}\|Ax-b\|^2,
		\]
		where $f$ is a convex $(L_0^f,L_1^f)$-smooth function and $\beta>0$. Define the sublevel set
		\[
		S:=\{x:\mathcal{L}(x)\leq \mathcal{L}(x_s)\},
		\]
		where $x_s$ is the starting point of the subproblem.
		
		Then $\mathcal{L}$ is $(L_0^{\mathcal{L}}, L_1^{f})$-smooth on $S$, where
		\begin{equation}\label{l_0def}
			L_0^{\mathcal{L}}
			=
			L_0^f
			+
			\beta\|A^\top A\|
			+
			L_1^f B_S,
		\end{equation}
		with
		\[
		B_S
		=
		\|A\|
		\bigl(
		Y
		+
		\beta R_S
		\bigr),
		\]
		and
		\[
		R_S
		=
		\sqrt{
			\frac{2(\mathcal{L}(x_s)-f^\star)}{\beta}
			+
			\frac{Y^2}{\beta^2}
		}
		+
		\frac{Y}{\beta}.
		\]
		and $Y = 2\|y^\star\| + \sqrt{C_\varepsilon}$.
	\end{proposition}
	
	See proof in Appendix~\ref{sect:ap:prop:l0l1_obj}.
	
	Thus, \eqref{eq:augmented_lagrangian} possesses finite $L_0$ and $L_1$-smoothness constants throughout the execution of Algorithm~\ref{alg:iALM}. These bounds will be used in the complexity analysis of the inner optimization procedure.
	
	We are now ready to estimate the iteration complexity of the overall scheme described in Algorithm~\ref{alg:iALM}.
	
	\subsection{Ergodic convergence rate analysis of overall iterations}
	
	We now derive the overall iteration complexity of the proposed method. Given the choices of $\{\rho_k\}$ and $\{\beta_k\}$ specified in \eqref{beta_rho_pick}, our next goal is to determine an appropriate sequence of inner accuracies $\{\varepsilon_k\}$.
	
	To this end, we recall the complexity estimate for the Two-Stage Acceleration Procedure (Algorithm~\ref{alg:two_stage}) established in \cite{vankov2025optimizing}.
	
	\begin{theorem}[See proof in \cite{vankov2025optimizing}]
		\label{th:vankov}
		Let $f$ be a convex $(L_0,L_1)$-smooth function, let $x^\star$ be a minimizer of $f$, and let $R = \|x_0 - x^\star\|$.
		Then the Two-Stage Acceleration Procedure computes a point $x$ satisfying $f(x)-f(x^\star)\leq \varepsilon$
		using at most
		\[
		(m+1)
		\left\lceil
		\sqrt{\frac{5L_0R^2}{\varepsilon}}
		\right\rceil
		+
		\left\lceil
		\left(13+\frac{18}{e}\right) L_1^2R^2
		\right\rceil
		\]
		oracle calls. Here $m \geq 1$ is the number of oracle queries of inner iterations of AGMsDR algorithm.
	\end{theorem}
	
	It is important to note that Proposition~\ref{prop:l0l1_obj} is satisfied on subset $S:=\{x:\mathcal{L}(x)\leq \mathcal{L}(x_s)\}$ so we need monotonically decreasing Algorithm~\ref{alg:two_stage}. For that purpose we pick monotonically decreasing step size in AGMsDR algorithm $x^{k+1} = T(y_k)$.
	
	Now we can proof the main iteration complexity estimate.
	
	\begin{theorem}[Ergodic iteration complexity, Setting~\ref{beta_rho_pick}]\label{th:iter_complexity_res}
		Suppose that Assumption \ref{assum:conv} holds. For any given $\varepsilon>0$, let $K$ be a positive integer. Let $\{\beta_k\}$ and $\{\rho_k\}$ is constant sequence chosen according to \eqref{beta_rho_pick}, and let $\{\varepsilon_k\}$ be defined as in \eqref{rho_epsilon_pick}. Then, for problem \ref{eq:main_problem}, the following estimates hold:
		
		\begin{equation}\label{th:it_comp_f}
			\left|
			f_0\left(\overline{\mathbf{x}}^K\right)-f_0\left(\mathbf{x}^*\right)
			\right| 
			\leq \frac{\varepsilon\left(2\left\|\mathbf{y}^*\right\|^2\right)}{C_\beta}+\frac{\varepsilon}{2} \frac{C_{\varepsilon}}{C_\beta}
		\end{equation}
		
		\begin{equation}\label{th:it_comp_cnstr}
			\| \mathbf{A} \overline{\mathbf{x}}^K-\mathbf{b} \|
			\leq \frac{\varepsilon \left(1+\left\|\mathbf{y}^*\right\|\right)^2}{2 C_\beta}+\frac{\varepsilon}{2} \frac{C_{\varepsilon}}{C_\beta}
		\end{equation}
		
		Moreover, Algorithm~\ref{alg:iALM} generates $\overline{x}^K$ defined in \eqref{ergodic_point_def} after evaluating the gradients of $f$ at most $T_K$ times, where
		
		\begin{align*}
			T_K = (m+1) K R \sqrt{\frac{C_{\beta_k} 10 L_0^{\mathcal{L}} }{C_\varepsilon \ \varepsilon}} + c K (L_1^f R)^2
		\end{align*}
		
		, where $R = \max_{x, y} \| x - y \|, x, y \in Q$, $c = 13+\frac{18}{e}$, and according to Proposition~\eqref{prop:l0l1_obj} $L_0^{\mathcal{L}} = L_0^f + \beta_k\|A^\top A\| + L_1^f B_S$.
	\end{theorem}
	
	See proof in Appendix~\ref{sect:ap:th:iter_complexity_res}.
	
	Let us plug in the definition of the $L_0^{\mathcal{L}}$ and derive the asymptotic iteration complexity of the Algorithm~\ref{alg:iALM}.
	
	\begin{remark}[Asymptotic iteration complexity]\label{remark:assym_beta}
		Suppose that \(C_\beta\), \(C_\varepsilon\), \(K\), \(R\), \(L_0^f\), \(L_1^f\), \(\|A\|\), \(\|A^\top A\|\), and \(\|y^*\|\) are independent of the target accuracy \(\varepsilon\). Then
		\[
		L_0^{\mathcal{L}}
		=
		L_0^f
		+
		\frac{C_\beta}{K\varepsilon}\|A^\top A\|
		+
		L_1^f\|A\|
		\left(
		2M_y
		+
		\sqrt{
			\frac{2C_\beta}{K\varepsilon}
			\bigl(\mathcal L(x_s)-f^*\bigr)
			+
			M_y^2
		}
		\right),
		\]
		satisfies
		\[
		L_0^{\mathcal{L}}
		=
		\mathcal O\!\left(\frac{1}{\varepsilon}\right),
		\qquad
		\varepsilon\to0.
		\]
		
		Consequently, the total number of gradient evaluations obtained in
		Theorem~\ref{th:iter_complexity_res} obeys
		\[
		T_K
		=
		(m+1)KR
		\sqrt{
			\frac{10C_\beta L_0^{\mathcal{L}}}
			{C_\varepsilon\varepsilon}
		}
		+
		cK(L_1^fR)^2
		=
		\mathcal O\!\left(\frac{K}{\varepsilon}\right)
		+
		\mathcal O(K).
		\]
		
		Therefore, for fixed \(K\), the overall oracle complexity scales as
		\[
		T_K
		=
		\mathcal O\!\left(\frac{1}{\varepsilon}\right),
		\qquad
		\varepsilon\to0.
		\]
	\end{remark}
	
	It is worth noting that the choice of the number of outer iterations \(K\) induces a trade-off between the theoretical complexity bound and the conditioning of the augmented Lagrangian subproblems. Indeed, the leading term in the oracle complexity estimate scales as \(\mathcal{O}(\sqrt{K}/\varepsilon)\), suggesting that smaller values of \(K\) are preferable from the viewpoint of the asymptotic convergence rate. However, since the penalty parameters are chosen as \(\rho_k=\beta_k=C_\beta/(K\varepsilon)\), decreasing \(K\) simultaneously increases the penalty coefficient and consequently enlarges the smoothness constant \(L_0^{\mathcal{L}}\) of the augmented Lagrangian. In particular, the choice \(K=1\) yields the most favorable asymptotic complexity bound but also produces the largest penalty parameter, potentially resulting in poorly conditioned inner subproblems. \textit{Thus, \(K\) acts as a tuning parameter balancing the number of outer ALM iterations against the conditioning of the inner optimization problems}.
	
	We also see greater $\beta_k$ parameter involve greater $L_0^{\mathcal{L}}$ parameter which increases radius of the set $Q$ where we start accelerated gradient method. But in other side big $L_0^{\mathcal{L}}$ constant make fast gradient method slower because convergence rate of the AGMsDR is $O(\frac{2 L_0^{\mathcal{L}} R^2}{k^2})$ so it is for not free
	
	\section{Details on Penalty Parameter Selection}\label{sect:details_pen_par}
	In this section, we aim to derive an explicit expression for the constant $C_\beta$ in \eqref{beta_rho_pick} and also we show that this penalty sequence minimizes the upper bound on the total number of inner iterations derived in \eqref{conv_rate}. We also show that under increasing penalty parameter our complexity result remains true. We start with the next simple lemma
	
	\begin{lemma}\label{lem:fun_convexity}
		Let
		\[
		g(x)=xL_0^{\mathcal L}(x),
		\]
		where $L_0^{\mathcal L}$ is defined in \eqref{l_0def}. Then $g$ is a convex and monotonically increasing function on $\mathbb{R}_+$.
	\end{lemma}
	
	See proof in Appendix~\ref{sect:ap:fun_convexity}.
	
	\begin{proposition}\label{prop:optimal_beta}
		Let us assume all assumptions from the Theorem~\ref{th:iter_complexity_res}. Then, for problem \ref{eq:main_problem}, the following $\{ \beta_k \}$ sequence are optimal choice to make iteration counts minimal:
		\begin{equation*}
			\beta_0 = \beta_1 = ... = \beta_{K-1} = \frac{C_\beta}{K \varepsilon},
		\end{equation*}
		where $C_\beta = 2 \| y^* \|^2 + C_\varepsilon / 2$.
	\end{proposition}
	
	As a result, we have established that the constant penalty sequence

	\[
	\boxed{\beta_k=\frac{2\|y^*\|^2+C_\varepsilon/2}{K\varepsilon}}
	\]
	
	minimizes the iteration complexity of Algorithms~\ref{alg:iALM} and \ref{alg:two_stage} applied to the inner subproblems.
	
	Nevertheless, a constant penalty parameter may not be a good practical choice because it can lead to ill-conditioned subproblems. Therefore, we also consider alternative increasing penalty sequences. Note that, in this case, the smoothness constant $L_0^{\mathcal L}(\beta_k)$ may vary from one iteration to another.
	
	Following the geometric penalty strategy proposed in \cite{xu2021iteration}, we consider the penalty sequence
	
	\begin{equation}\label{increasing_penalty}
		\boxed{
			\beta_k
			=
			\beta_g\sigma^k
		}
		\quad \quad , \quad \quad 
		\beta_g =
		\frac{C_\beta}{\varepsilon}
		\frac{\sigma-1}{\sigma^K-1}
	\end{equation}
	
	\begin{theorem}[Ergodic iteration complexity with a geometrically increasing penalty]\label{th:increasing_penalty}
		Suppose that all assumptions of Theorem~\ref{th:iter_complexity_res} hold. Let the penalty and proximal parameters satisfy
		\[
		\{\beta_k\}=\{\rho_k\},
		\]
		where $\{\beta_k\}$ is defined by \eqref{increasing_penalty}, and let the inner accuracies $\{\varepsilon_k\}$ be chosen according to \eqref{rho_epsilon_pick}. Then Algorithm~\ref{alg:iALM}, applied to problem~\eqref{eq:main_problem}, generates the ergodic point $\overline{x}^K$ defined in \eqref{ergodic_point_def} after at most $T_K$ gradient evaluations of $f$, where
		\[
		T_K
		\le
		(m+1)
		KR
		\sqrt{
			\frac{10L_0^{\mathcal L}(\beta_{\max})C_\beta}
			{C_\varepsilon\varepsilon}
			\frac{(\sigma-1)\sigma^{K-1}}
			{\sigma^K-1}
		}
		+
		cK(L_1^fR)^2,
		\]
		with
		\[
		\beta_{\max}
		=
		\frac{C_\beta}{\varepsilon}
		\frac{\sigma-1}{\sigma^K-1}
		\sigma^{K-1}.
		\]
	\end{theorem}
	
	See proof in Appendix~\ref{ap:th:increasing_penalty}.
	
	By Remark~\ref{remark:assym_beta}, replacing the constant penalty parameter with $\beta_{\max}$ yields the following oracle complexity estimate for the algorithm with a geometrically increasing penalty sequence:
	\[
	T_K
	=
	\mathcal O\!\left(\frac{1}{\varepsilon}\right),
	\qquad
	\varepsilon\to0,
	\]
	for fixed $K$.
	
	\section{Experimental Setup}
	
	We consider the convex optimization problem with affine equality constraints
	$$
	\min_{x\in\mathbb{R}^m} f(x)
	=
	\frac{1}{p}\|x\|^p
	\quad
	\text{s.t.}
	\quad
	Ax=b,
	$$
	where $p>2$, $A\in\mathbb{R}^{n\times m}$, and $b\in\mathbb{R}^n$.
	
	The objective function $f$ is strictly convex for $p>1$. Hence, whenever the feasible set is nonempty, the problem admits a unique optimal solution, although the feasible set itself may contain infinitely many points. The function $f$ is not globally $L$-smooth, but it satisfies the $(L_0,L_1)$-smoothness condition with
	$$
	L_0^f=p,
	\qquad
	L_1^f=p-1,
	$$
	see, e.g.,~\cite{gorbunov2024methods}.
	
	\subsection{Augmented Lagrangian Function}
	
	For the numerical experiments, we consider the augmented Lagrangian
	$$
	F(x)
	=
	\frac{1}{p}\|x\|^p
	+
	y^\top(Ax-b)
	+
	\frac{\beta}{2}\|Ax-b\|^2,
	$$
	where $\beta>0$ is the penalty parameter and $y$ is the dual variable.
	
	The function $F$ is convex but is not globally $L$-smooth. On a bounded set $\|x\|\leq R$, its gradient is Lipschitz continuous. For the first term,
	$$
	L_f(R)=(p-1)R^{p-2},
	$$
	while for the quadratic penalty term,
	$$
	L_{\mathrm{pen}}
	=
	\beta\|A^\top A\|_2
	=
	\beta\|A\|_2^2.
	$$
	Therefore, on the ball $\|x\|\leq R$,
	$$
	\boxed{
		L=(p-1)R^{p-2}+\beta\|A\|_2^2.
	}
	$$
	
	For a monotone gradient method, $F(x_k)\leq F(x_0)$ for all $k$. Since the quadratic penalty term is nonnegative,
	$$
	\frac{1}{p}\|x_k\|^p
	\leq
	F(x_k)
	\leq
	F(x_0),
	$$
	which gives
	$$
	\|x_k\|
	\leq
	\left(pF(x_0)\right)^{1/p}.
	$$
	Thus, we use
	$$
	\boxed{
		R=\left(pF(x_0)\right)^{1/p}.
	}
	$$
	
	For the $(L_0,L_1)$-smoothness parameters of the augmented Lagrangian, we use
	$$
	L_1=L_1^f,
	$$
	and
	$$
	L_0
	=
	L_0^f
	+
	\beta\|A\|_2^2
	+
	L_1^f\|A\|_2
	\left(
	2Y+\sqrt{2\beta\Delta+Y^2}
	\right),
	$$
	where
	$$
	Y=2\|y^*\|+\sqrt{C_\varepsilon},
	\qquad
	\Delta=\mathcal{L}(x)-\mathcal{L}^*.
	$$
	
	\subsection{Test Problem Generation}
	
	The matrix $A$ is generated such that the system $Ax=b$ has infinitely many solutions. A feasible point $x_{\mathrm{feas}}$ is sampled randomly, and the right-hand side is constructed as
	$$
	b=Ax_{\mathrm{feas}}.
	$$
	The optimal solution of the resulting problem can then be computed analytically as the minimum-norm solution of the linear system:
	$$
	\boxed{
		x^*
		=
		A^\top(AA^\top)^{-1}b.
	}
	$$
	
	\subsection{Algorithmic Parameters}
	
	The following parameters are used for the iALM procedure:
	$$
	C_\varepsilon=C_\beta=2,
	\qquad
	\varepsilon=10^{-3},
	\qquad
	K=10,
	\qquad
	\varepsilon_k=\varepsilon.
	$$
	Thus, each inner problem is solved to accuracy $\varepsilon$.
	
	We compare iALM Algorithm~\ref{alg:iALM} with the following three methods on the subproblems:
	\begin{enumerate}
		\item the standard gradient method with the Lipschitz constant $L$;
		\item the gradient method based on the $(L_0,L_1)$-smoothness parameters;
		\item the Two-Stage Acceleration Procedure~\ref{alg:two_stage}, with AGMsDR~\ref{alg:agmsdr} used at the second stage.
	\end{enumerate}
	
	The convergence of all three methods toward the KKT point is shown in Figure~\ref{fig:convergence_kkt}.
	
	\begin{figure}[t]
		\centering
		\includegraphics[width=1.0\textwidth]{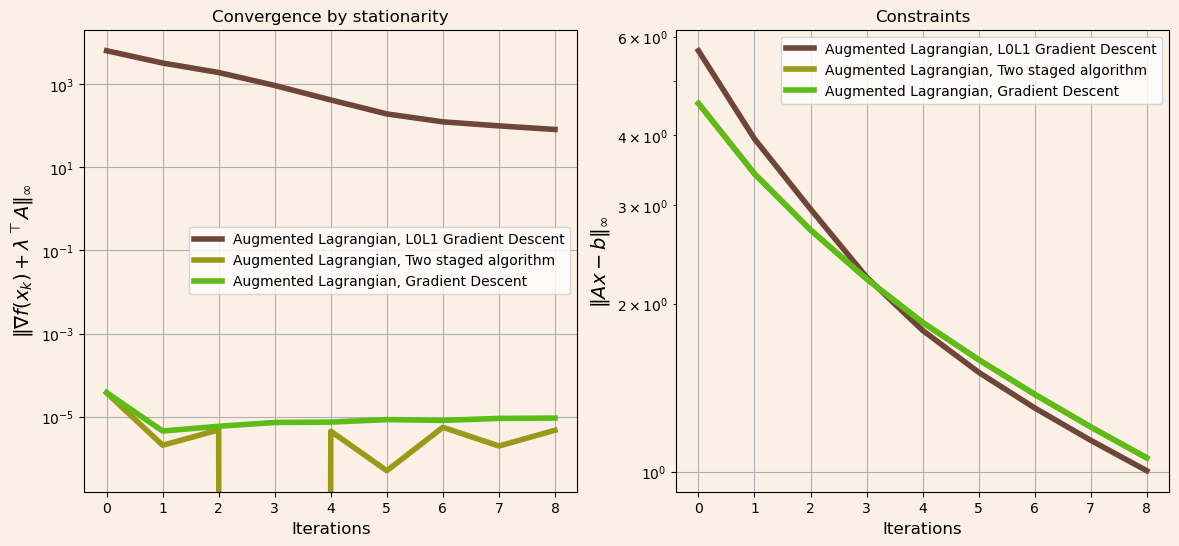}
		\caption{Convergence of the considered methods to the KKT point. On X axis of both graphs there is the outer iterations.}
		\label{fig:convergence_kkt}
	\end{figure}
	
	\subsection{Experimental Results}
	
	The standard gradient method with the estimated Lipschitz constant $L$ provides the fast and accurate convergence among the considered methods for this test problem, neverthless Two-Stage Acceleration Procedur demonstrates more accurate solution at each outer iterations. It is worth to note that when $L_1=0$, the $(L_0,L_1)$-based method exhibits behavior very similar to the standard gradient method, which is consistent with the transition to the classical smoothness setting.

	\section{Conclusion}
	
	In this paper, we studied the application of the inexact Augmented Lagrangian Method (iALM) for linearly constrained convex optimization problems with $(L_0,L_1)$-smooth objective functions. Unlike the classical smoothness setting, the augmented Lagrangian is generally not globally Lipschitz-smooth, which prevents the direct application of standard accelerated first-order methods.
	
	We showed that the augmented Lagrangian preserves the $(L_0,L_1)$-smooth structure of the original objective function, with parameters depending on the penalty coefficient. This property allows us to employ the Two-Stage Acceleration Procedure for solving the primal subproblems of iALM. As a result, we established the global ergodic iteration complexity of the proposed approach and obtained guarantees on the number of first-order oracle evaluations required to achieve a given accuracy.
	
	The obtained results extend existing complexity analysis of augmented Lagrangian methods beyond the classical Lipschitz-gradient framework and provide a theoretical foundation for applying iALM to a broader class of generalized smooth convex problems.
	
	Future work includes extending the proposed framework to inequaltiy constraint and deriving nonergodic convergence rate.
	
	%
	
	\bibliographystyle{splncs04}
	\bibliography{iALM_l0l1}

@article{zhang2019gradient,
	title={Why gradient clipping accelerates training: A theoretical justification for adaptivity},
	author={Zhang, Jingzhao and He, Tianxing and Sra, Suvrit and Jadbabaie, Ali},
	journal={arXiv preprint arXiv:1905.11881},
	year={2019}
}

@inproceedings{vankov2025optimizing,
	title={Optimizing $(L\_0, L\_1) $-Smooth Functions by Gradient Methods},
	author={Vankov, Daniil and Rodomanov, Anton and Nedich, Angelia and Sankar, Lalitha and Stich, Sebastian},
	booktitle={International Conference on Learning Representations},
	volume={2025},
	pages={15953--15979},
	year={2025}
}

@article{gorbunov2024methods,
	title={Methods for convex $(l\_0, l\_1) $-smooth optimization: Clipping, acceleration, and adaptivity},
	author={Gorbunov, Eduard and Tupitsa, Nazarii and Choudhury, Sayantan and Aliev, Alen and Richt{\'a}rik, Peter and Horv{\'a}th, Samuel and Tak{\'a}{\v{c}}, Martin},
	journal={arXiv preprint arXiv:2409.14989},
	year={2024}
}

@article{zhang2020improved,
	title={Improved analysis of clipping algorithms for non-convex optimization},
	author={Zhang, Bohang and Jin, Jikai and Fang, Cong and Wang, Liwei},
	journal={Advances in Neural Information Processing Systems},
	volume={33},
	pages={15511--15521},
	year={2020}
}

@inproceedings{koloskova2023revisiting,
	title={Revisiting gradient clipping: Stochastic bias and tight convergence guarantees},
	author={Koloskova, Anastasia and Hendrikx, Hadrien and Stich, Sebastian U},
	booktitle={International Conference on Machine Learning},
	pages={17343--17363},
	year={2023},
	organization={PMLR}
}

@inproceedings{chen2023generalized,
	title={Generalized-smooth nonconvex optimization is as efficient as smooth nonconvex optimization},
	author={Chen, Ziyi and Zhou, Yi and Liang, Yingbin and Lu, Zhaosong},
	booktitle={International Conference on Machine Learning},
	pages={5396--5427},
	year={2023},
	organization={PMLR}
}

@article{li2023convex,
	title={Convex and non-convex optimization under generalized smoothness},
	author={Li, Haochuan and Qian, Jian and Tian, Yi and Rakhlin, Alexander and Jadbabaie, Ali},
	journal={Advances in Neural Information Processing Systems},
	volume={36},
	pages={40238--40271},
	year={2023}
}

@article{lobanov2024linear,
	title={Linear Convergence Rate in Convex Setup is Possible! Gradient Descent Method Variants under $(L\_0, L\_1) $-Smoothness},
	author={Lobanov, Aleksandr and Gasnikov, Alexander and Gorbunov, Eduard and Tak{\'a}{\v{c}}, Martin},
	journal={arXiv preprint arXiv:2412.17050},
	year={2024}
}

@article{vyguzov2025frank,
	author  = {Vyguzov, A. A. and Stonyakin, F. S. and Gasnikov, A. V.},
	title   = {Frank-Wolfe algorithms for $(L_0, L_1)$-smooth functions},
	journal = {Journal of Nonlinear and Variational Analysis},
	volume  = {10},
	year    = {2026},
	pages   = {597--616}
}

@article{xu2021iteration,
	title={Iteration complexity of inexact augmented Lagrangian methods for constrained convex programming},
	author={Xu, Yangyang},
	journal={Mathematical Programming},
	volume={185},
	number={1},
	pages={199--244},
	year={2021},
	publisher={Springer}
}

@article{tyurin2025near,
	title={Near-Optimal Convergence of Accelerated Gradient Methods under Generalized and $(L\_0, L\_1) $-Smoothness},
	author={Tyurin, Alexander},
	journal={arXiv preprint arXiv:2508.06884},
	year={2025}
}

@article{nesterov2021primal,
	title={Primal--dual accelerated gradient methods with small-dimensional relaxation oracle},
	author={Nesterov, Yurii and Gasnikov, Alexander and Guminov, Sergey and Dvurechensky, Pavel},
	journal={Optimization Methods and Software},
	volume={36},
	number={4},
	pages={773--810},
	year={2021},
	publisher={Taylor \& Francis}
}

@article{boyd2011distributed,
	title={Distributed optimization and statistical learning via the alternating direction method of multipliers},
	author={Boyd, Stephen and Parikh, Neal and Chu, Eric and Peleato, Borja and Eckstein, Jonathan},
	journal={Foundations and Trends{\textregistered} in Machine learning},
	volume={3},
	number={1},
	pages={1--122},
	year={2011},
	publisher={Emerald Publishing Limited}
}
	
	\newpage
	
	\section{Missing Proofs in Section~\ref{sect:erg_conv_rate_analyzes}}
	
	\subsection{Proof of Proposition~\ref{prop:l0l1_obj}}\label{sect:ap:prop:l0l1_obj}
	
	\begin{proof}
		By completing the square, the augmented Lagrangian can be rewritten as
		\[
		\mathcal{L}(x)
		=
		f(x)
		+
		\frac{\beta}{2}
		\left\|
		Ax-b+\frac{y_k}{\beta}
		\right\|^2
		-
		\frac{\|y_k\|^2}{2\beta}.
		\]
		
		Since $f(x)\geq f^\star$, for every $x\in S$ we have
		\[
		f^\star
		+
		\frac{\beta}{2}
		\left\|
		Ax-b+\frac{y_k}{\beta}
		\right\|^2
		-
		\frac{\|y_k\|^2}{2\beta}
		\leq
		\mathcal{L}(x_s).
		\]
		
		Rearranging the above inequality yields
		\[
		\|Ax-b\|
		\leq
		\sqrt{
			\frac{2(\mathcal{L}(x_s)-f^\star)}{\beta}
			+
			\frac{\|y_k\|^2}{\beta^2}
		}
		+
		\frac{\|y_k\|}{\beta}
		=:R_S.
		\]
		
		Therefore, for every $x\in S$,
		\[
		\|A^\top y_k+\beta A^\top(Ax-b)\|
		\leq
		\|A\|
		\bigl(
		\|y_k\|
		+
		\beta R_S
		\bigr)
		=:B_S.
		\]
		
		Next, we invoke Lemma~\ref{lem:y_bnd}, which provides the uniform bound $\| y_k \| \leq Y$.
		
		Since $f$ is $(L_0^f,L_1^f)$-smooth, it satisfies
		\[
		\|\nabla^2 f(x)\|
		\leq
		L_0^f
		+
		L_1^f\|\nabla f(x)\|.
		\]
		
		Furthermore,
		\[
		\nabla \mathcal{L}(x)
		=
		\nabla f(x)
		+
		A^\top y
		+
		\beta A^\top(Ax-b),
		\]
		and therefore
		\[
		\|\nabla f(x)\|
		\leq
		\|\nabla \mathcal{L}(x)\|
		+
		\|A^\top y+\beta A^\top(Ax-b)\|
		\leq
		\|\nabla \mathcal{L}(x)\|
		+
		B_S.
		\]
		
		On the other hand,
		\[
		\nabla^2\mathcal{L}(x)
		=
		\nabla^2f(x)
		+
		\beta A^\top A.
		\]
		
		Combining the above estimates gives
		\[
		\|\nabla^2\mathcal{L}(x)\|
		\leq
		L_0^f
		+
		\beta\|A^\top A\|
		+
		L_1^f B_S
		+
		L_1^f\|\nabla\mathcal{L}(x)\|.
		\]
		
		Hence,
		\[
		\|\nabla^2\mathcal{L}(x)\|
		\leq
		L_0^{\mathcal{L}}
		+
		L_1^f\|\nabla\mathcal{L}(x)\|,
		\qquad x\in S,
		\]
		which proves that $\mathcal{L}$ is $(L_0^{\mathcal{L}}, L_1^{f})$-smooth on $S$.
	\end{proof}
	
	\subsection{Proof of Theorem~\ref{th:iter_complexity_res}}\label{sect:ap:th:iter_complexity_res}
	
	\begin{proof}
		Applying Theorem~\ref{th:vankov} with
		\(
		c = 13+\frac{18}{e}
		\)
		to each augmented Lagrangian subproblem \eqref{eq:augmented_lagrangian}, the total number of oracle calls required by Algorithm~\ref{alg:iALM} is bounded by
		\[
		\sum_{k=0}^{K-1}
		\left(
		(m+1)
		\sqrt{\frac{5 L_{0,k} R^2}{\varepsilon_k}}
		+
		c (L_1^f R)^2
		\right).
		\]
		
		Combining Proposition~\ref{prop:l0l1_obj} and Lemma~\ref{lem:y_bnd}, we conclude that \(L_{0,k}\) is uniformly bounded, i.e.,
		\[
		L_{0,k}\leq L_0^{\mathcal{L}},
		\qquad k=0,\ldots,K-1.
		\]
		Therefore,
		\[
		\sum_{k=0}^{K-1}T_k
		\leq
		\sum_{k=0}^{K-1}
		(m+1)
		\sqrt{\frac{5L_0^{\mathcal{L}} R^2}{\varepsilon_k}}
		+
		c K (L_1^f R)^2.
		\]
		
		It remains to choose the sequence \(\{\varepsilon_k\}\). Since \(\rho_k\) and \(\varepsilon_k\) are chosen according to \eqref{rho_epsilon_pick}, we arrive at the optimization problem
		\[
		\min_{\varepsilon_k>0}
		\;
		\sum_{k=0}^{K-1}
		(m+1)
		\sqrt{\frac{5L_0^{\mathcal{L}}R^2}{\varepsilon_k}}
		\qquad
		\text{s.t.}
		\qquad
		\sum_{k=0}^{K-1}\rho_k\varepsilon_k
		\leq
		\frac{C_\varepsilon}{2}.
		\]
		
		Using the KKT optimality conditions, the optimal solution is given by
		\[
		\varepsilon_k^*
		=
		\frac{C_\varepsilon}{2}
		\,
		\frac{\rho_k^{-2/3}}
		{\sum_{j=0}^{K-1}\rho_j^{1/3}}.
		\]
		
		Substituting \(\varepsilon_k^*\) into the objective yields
		\begin{equation}\label{conv_rate}
			\sum_{k=0}^{K-1}T_k
			\leq
			(m+1)
			\sqrt{\frac{10L_0^{\mathcal{L}}R^2}{C_\varepsilon}}
			\left(
			\sum_{k=0}^{K-1}\rho_k^{1/3}
			\right)^{3/2}
			+
			c K (L_1^f R)^2.
		\end{equation}
		
		Using the choice
		\[
		\rho_k=\frac{C_\beta}{K\varepsilon},
		\qquad k=0,\ldots,K-1,
		\]
		we obtain
		\[
		\sum_{k=0}^{K-1}\rho_k^{1/3}
		=
		K
		\left(
		\frac{C_\beta}{K\varepsilon}
		\right)^{1/3}
		=
		K^{2/3}
		\left(
		\frac{C_\beta}{\varepsilon}
		\right)^{1/3}.
		\]
		Hence,
		\[
		\left(
		\sum_{k=0}^{K-1}\rho_k^{1/3}
		\right)^{3/2}
		=
		K
		\sqrt{\frac{C_\beta}{\varepsilon}}.
		\]
		
		Substituting the above expression into the complexity estimate gives
		\[
		\sum_{k=0}^{K-1}T_k
		\leq
		(m+1)
		K R
		\sqrt{
			\frac{10\,C_\beta\,L_0^{\mathcal{L}}}
			{C_\varepsilon\,\varepsilon}
		}
		+
		cK(L_1^fR)^2,
		\]
		which is exactly the claimed bound.
	\end{proof}
	
	\section{Missing Proofs in Section~\ref{sect:details_pen_par}}
	
	\subsection{Proof of Lemma~\ref{lem:fun_convexity}}\label{sect:ap:fun_convexity}
	
	\begin{proof}[Proof of Lemma~\ref{lem:fun_convexity}]
		We consider $g(x) = x L_0^{\mathcal{L}}(x)$, where 
		$$
		L_0^{\mathcal{L}}
		=
		L_0^f
		+
		\beta\|A^\top A\|
		+
		L_1^f B_S,
		$$
		with
		$B_S=\|A\|\bigl(\|y^*\|+\beta R_S\bigr)$
		and
		$R_S=\sqrt{\frac{2(\mathcal{L}(x_s)-f^\star)}{\beta}+\frac{\|y^*\|^2}{\beta^2}}+\frac{\|y^*\|}{\beta}.$
		
		Substituting the definition of $L_0^{\mathcal{L}}$ and denoting $\Delta = \mathcal{L}(x_s)-f^\star$ we get:
		
		$$
		g(x) = x \left( \left( L_0^f + 2 L_1^f \| A \| \| y^* \| \right) + \| A^\top A\| x + L_1^f \| A \| x \left(\sqrt{\frac{2 \Delta}{x} + \frac{\| y^* \|^2}{x^2} } + \frac{\| y^* \|}{x} \right) \right)
		$$
		
		Denoting $a_0 = L_0^f + 2 L_1^f \| A \| \| y^* \|, a_1 = \| A^\top A\|, a_2 = L_1^f \| A \|, a_3 = \| y^* \|^2, a_4 = 2 \Delta$ (notice that $a_i \geq 0$) we get
		\begin{equation}\label{lem:fun_convexity:eq_a}
			g(x) = a_1 x^2 + a_0 x + a_2 x \sqrt{a_4 x + a_3}
		\end{equation}
		
		By the definitions of the constants, we have $a_i \geq 0$ for all $i \in \{0, 1, 2, 3, 4\}$. 
		If $a_3 = a_4 = 0$, the function reduces to the quadratic polynomial $g(x) = a_1 x^2 + a_0 x$, which is trivially convex and monotonically increasing on $\mathbb{R}_{+}$. 
		Otherwise, $a_4 x + a_3 > 0$ for all $x > 0$. To establish the convexity and monotonicity of $g(x)$ on $\mathbb{R}_{+}$, we analyze its first and second derivatives on $\mathbb{R}_{++}$.
		
		Let $h(x) = x \sqrt{a_4 x + a_3}$. The first derivative of $h(x)$ is given by:
		\begin{equation*}
			h'(x) = \sqrt{a_4 x + a_3} + \frac{a_4 x}{2\sqrt{a_4 x + a_3}} = \frac{3a_4 x + 2a_3}{2\sqrt{a_4 x + a_3}}.
		\end{equation*}
		Since $a_3, a_4 \geq 0$ and $x > 0$, it follows that $h'(x) \geq 0$. 
		The first derivative of $g(x)$ is then:
		\begin{equation*}
			g'(x) = 2a_1 x + a_0 + a_2 h'(x).
		\end{equation*}
		As $a_0, a_1, a_2 \geq 0$, we conclude that $g'(x) \geq 0$ for all $x \in \mathbb{R}_{++}$. Since $g(x)$ is continuous on $\mathbb{R}_{+}$, this implies that $g(x)$ is monotonically increasing on $\mathbb{R}_{+}$.
		
		Next, we compute the second derivative of $h(x)$:
		\begin{equation*}
			h''(x) = \frac{a_4}{2\sqrt{a_4 x + a_3}} + \frac{a_4}{2\sqrt{a_4 x + a_3}} - \frac{a_4^2 x}{4(a_4 x + a_3)^{3/2}} = \frac{a_4}{(a_4 x + a_3)^{3/2}} \left( \frac{3a_4 x}{4} + a_3 \right).
		\end{equation*}
		Given that $a_3, a_4 \geq 0$ and $x > 0$, we have $h''(x) \geq 0$. 
		The second derivative of $g(x)$ is:
		\begin{equation*}
			g''(x) = 2a_1 + a_2 h''(x).
		\end{equation*}
		Since $a_1, a_2 \geq 0$ and $h''(x) \geq 0$, it follows that $g''(x) \geq 0$ for all $x \in \mathbb{R}_{++}$. 
		Because $g(x)$ is continuous on $\mathbb{R}_{+}$ and its second derivative is non-negative on $\mathbb{R}_{++}$, we conclude that $g(x)$ is convex on $\mathbb{R}_{+}$.
		
		This completes the proof.
	\end{proof}
	
	\subsection{Proof of Proposition~\ref{prop:optimal_beta}}
	
	\begin{proof}
		According to the proof of Theorem~\ref{th:iter_complexity_res}, after optimizing the sequence of inner accuracies $\{\varepsilon_k\}$, the total iteration complexity of all inner solvers is bounded by
		\[
		\sum_{k=0}^{K-1}T_k
		\leq
		(m+1)
		\sqrt{\frac{10R^2}{C_\varepsilon}}
		\left(
		\sum_{k=0}^{K-1}
		\beta_k^{1/3}
		\left(L_0^{\mathcal L}(\beta_k)\right)^{1/3}
		\right)^{3/2}
		+
		cK(L_1^fR)^2.
		\]
		Since the second term is independent of the penalty parameters $\{\beta_k\}$, minimizing the total iteration complexity is equivalent to solving
		\[
		\min_{\{\beta_k\}}
		\sum_{k=0}^{K-1}
		\left( \beta_k L_0^{\mathcal L}(\beta_k)\right)^{1/3},
		\]
		subject to
		\[
		\beta_k>0,\qquad k=0,\ldots,K-1,
		\]
		and
		\[
		\sum_{k=0}^{K-1}\beta_k
		\geq
		\frac{2\|y^*\|^2+C_\varepsilon/2}{\varepsilon}.
		\]
		
		Next, define the function
		\[
		g(x)=xL_0^{\mathcal L}(x).
		\]
		By Lemma~\ref{lem:fun_convexity}, the function $g(x)$ is strictly convex on $\mathbb{R}_{++}$. Since the function $f(x) = x^{1/3}$ is strictly increasing on $\mathbb{R}_{+}$, minimizing
		\[
		\sum_{k=0}^{K-1}
		\left(\beta_k L_0^{\mathcal L}(\beta_k)\right)^{1/3}
		\]
		is equivalent to minimizing
		\[
		\sum_{k=0}^{K-1}
		\beta_kL_0^{\mathcal L}(\beta_k).
		\]
		Hence, the optimization problem can be further simplified to
		\[
		\min_{\{\beta_k\}}
		\sum_{k=0}^{K-1}
		\beta_kL_0^{\mathcal L}(\beta_k),
		\]
		with the same constraints.
		
		To solve the above optimization problem, consider its Lagrangian
		\[
		\mathcal{L}(\{\beta_k\},\lambda)
		=
		\sum_{k=0}^{K-1}
		\beta_kL_0^{\mathcal L}(\beta_k)
		+
		\lambda
		\left(
		\frac{2\|y^*\|^2+C_\varepsilon/2}{\varepsilon}
		-
		\sum_{k=0}^{K-1}\beta_k
		\right),
		\]
		where $\lambda\geq0$ is the Lagrange multiplier associated with the inequality constraint.
		
		The first-order Karush--Kuhn--Tucker (KKT) optimality conditions imply that, for every
		$k=0,\ldots,K-1$,
		\[
		\frac{d}{d\beta_k}
		\left(
		\beta_kL_0^{\mathcal L}(\beta_k)
		\right)
		=
		\lambda.
		\]
		Since the function
		\[
		g(x)=xL_0^{\mathcal L}(x)
		\]
		is strictly convex by Lemma~\ref{lem:fun_convexity}, its derivative $g'(x)$ is strictly increasing. Consequently, the above stationarity conditions admit a unique solution, which satisfies
		\[
		\beta_0=\beta_1=\cdots=\beta_{K-1}.
		\]
		
		Furthermore, since $g(x)$ is increasing on $\mathbb{R}_+$ by Lemma~\ref{lem:fun_convexity}, the objective function is increasing with respect to each $\beta_k$. Therefore, the inequality constraint is active at the optimum, that is,
		\[
		\sum_{k=0}^{K-1}\beta_k
		=
		\frac{2\|y^*\|^2+C_\varepsilon/2}{\varepsilon}.
		\]
		Combining this equality with
		\[
		\beta_0=\beta_1=\cdots=\beta_{K-1},
		\]
		we obtain
		\[
		\beta_k
		=
		\frac{2\|y^*\|^2+C_\varepsilon/2}{K\varepsilon}
		=
		\frac{C_\beta}{K\varepsilon},
		\qquad
		k=0,\ldots,K-1,
		\]
		which completes the proof.
	\end{proof}
	
	\subsection{Proof of Proposition~\ref{th:increasing_penalty}}\label{ap:th:increasing_penalty}
	
	\begin{proof}
		According to Proposition~\ref{prop:optimal_beta}, after optimizing the sequence of inner accuracies $\{\varepsilon_k\}$, the total iteration complexity satisfies
		\[
		\sum_{k=0}^{K-1}T_k
		\le
		(m+1)
		\sqrt{\frac{10R^2}{C_\varepsilon}}
		\left(
		\sum_{k=0}^{K-1}
		\beta_k^{1/3}
		\left(L_0^{\mathcal L}(\beta_k)\right)^{1/3}
		\right)^{3/2}
		+
		cK(L_1^fR)^2.
		\]
		
		Let
		\[
		\beta_{\max}
		=
		\frac{C_\beta}{\varepsilon}
		\frac{\sigma-1}{\sigma^K-1}
		\sigma^{K-1}
		\]
		be the largest element of the increasing penalty sequence
		\eqref{increasing_penalty}. By Lemma~\ref{lem:fun_convexity}, the function
		\[
		g(\beta)=\beta L_0^{\mathcal L}(\beta)
		\]
		is monotonically increasing on $\mathbb{R}_+$. Hence,
		\[
		\beta_kL_0^{\mathcal L}(\beta_k)
		\le
		\beta_{\max}L_0^{\mathcal L}(\beta_{\max}),
		\qquad
		k=0,\ldots,K-1.
		\]
		Therefore,
		\[
		\sum_{k=0}^{K-1}
		\beta_k^{1/3}
		\left(L_0^{\mathcal L}(\beta_k)\right)^{1/3}
		\le
		K
		\left(
		\beta_{\max}L_0^{\mathcal L}(\beta_{\max})
		\right)^{1/3},
		\]
		and consequently
		\[
		\sum_{k=0}^{K-1}T_k
		\le
		(m+1)
		KR
		\sqrt{
			\frac{10\beta_{\max}L_0^{\mathcal L}(\beta_{\max})}
			{C_\varepsilon}
		}
		+
		cK(L_1^fR)^2.
		\]
		
		Finally, substituting the expression for $\beta_{\max}$ gives
		\[
		\beta_{\max}
		=
		\frac{C_\beta}{\varepsilon}
		\frac{\sigma-1}{\sigma^K-1}
		\sigma^{K-1},
		\]
		which yields
		\[
		T_K
		\le
		(m+1)
		KR
		\sqrt{
			\frac{10L_0^{\mathcal L}(\beta_{\max})C_\beta}
			{C_\varepsilon\varepsilon}
			\frac{(\sigma-1)\sigma^{K-1}}
			{\sigma^K-1}
		}
		+
		cK(L_1^fR)^2.
		\]
		This completes the proof.
	\end{proof}

\end{document}